\documentclass[12pt]{amsart}

\usepackage{amscd}
\usepackage{verbatim}

\advance\hoffset by -0.5 truecm

\usepackage{amssymb}

\newtheorem{Theorem}{Theorem}[section]

\newtheorem{Lemma}[Theorem]{Lemma}

\advance\hoffset by -0.5 truecm

\usepackage{amssymb}
\usepackage{graphicx}

\newtheorem{Example}[Theorem]{Example}

\newtheorem{Remark}[Theorem]{Remark}

\def\V{\mbox{Var}}

\def\R\re

\def\V{\bf V}

\def \re{{\mathbb R}}

\def \H{{\mathbb H}}

\def \0{\lambda_{0}}

\begin{document}
\hyphenation{Ya-ma-be co-rres-pon-ding hy-po-the-sis iso-pe-ri-me-tric gen-er-at-ed}
\title{Isoperimetric estimates in low dimensional Riemannian products}

\author[J. M. Ruiz]{Juan Miguel Ruiz}
 \address{Juan Miguel Ruiz, ENES UNAM \\
           37684 \\
          Le\'on. Gto. \\
          M\'exico.}
\email{mruiz@enes.unam.mx}

\author[A. V. Juarez]{Areli Vazquez Juarez}

\address{Areli V\'azquez Juarez, ENES UNAM \\
	37684 \\
	Le\'on. Gto. \\
	M\'exico.}
\email{areli@enes.unam.mx}

%\subjclass{53C25, 53C21, 58F17, 35J15}

\maketitle

\begin{abstract}  
Let $(T^k,h_k)=(S_{r_1}^1\times S_{r_2}^1 \times ... \times S_{r_k}^1, dt_1^2+dt_2^2+...+dt_k^2)$  be flat tori, $r_k\geq ...\geq r_2\geq r_1>0$ and $(\re^n,g_E)$ the Euclidean space with the flat metric. We compute  the isoperimetric profile of $(T^2\times \re^n, h_2+g_E)$, $2\leq n\leq 5$, for small    and big values of the volume. These computations give explicit lower bounds for the isoperimetric profile of $T^2\times\re^n$. We also note that similar estimates for $(T^k\times \re^n, h_k+g_E)$, $2\leq k\leq5$, $2\leq n\leq 7-k$, may be computed, provided estimates for $(T^{k-1}\times \re^{n+1}, h_{k-1}+g_E)$ exist. We compute this explicitly for $k=3$.  We use symmetrization techniques for product manifolds, based on work of A. Ros (\cite{Ros}) and F. Morgan (\cite{Morgan}).
	
\end{abstract}

\maketitle

\section{Introduction}

The isoperimetric problem is a classical question in differential geometry. An isoperimetric region of volume $t$, $0<t<V^n(M)$, in  a manifold $(M^n,g)$,  is a closed region $\Omega$ of volume $V^n(\Omega)=t$, such that its boundary area is minimal among the compact hypesurfaces $\Sigma \subset M$ enclosing a region of volume $t$. Throughout the article, the volume of a closed region $\Omega \subset M^n$ will mean $n-$dimensional Riemannian measure of $\Omega$ and  we will refer to them as $V^{n}(\Omega)$. On the other hand,  the area of a closed region $\Omega$ in the manifold $M^n$ will mean the $(n-1)-$dimensional Riemannian measure of $\partial \Omega$ and  and we will denote it by $V^{n-1}(\partial \Omega)$.
  
Given a Riemannian manifold $(M^n,g)$ of volume $V$, the isoperimetric function or isoperimetric profile of $(M,g)$ is the function $I_{(M,g)} : (0,V) \rightarrow (0,\infty )$ given
by 

$$I_{(M,g)} (t) = \inf \{ V^{n-1}( \partial U ) : V^n (U) = t , U \subset M^n, U \mathit{ \ a \ closed \  region} \} .$$

Note that the isoperimetric profile may be defined for manifolds of infinite volume. We will simply write $I_{M}$ when the metric $g$ is understood from context.  A  more detailed treatment of this subject may be found in  \cite{Ros}.

Although   a classical problem, the  isoperimetric profile is known for very few manifolds. It is known explicitly, for example, for space forms $(\re^n,g_E)$, $(S^n, g_0)$, $(\H^n,g_H)$, where $g_E, g_0$ and $g_H$ are the Euclidean, the round and the  hyperbolic metrics, respectively.  Other  examples include cylinders of the type $(S^n\times \re, g_0+dt^2)$ by the work of R. Pedrosa \cite{Pedrosa}, and for the Riemannian product of a low dimensional space form  with $S^1$, i.e., $(S^1\times \re^n,dt^2+g_E)$, $(S^1\times S^n, dt^2+g_0)$, $(S^1\times\H^n,dt^2+g_H)$ ($2 \leq n\leq 7$), by the work of R. Pedrosa and M. Ritor\'e \cite{Pedrosa2}.  Other results in this direction include lower bounds for isoperimetric profiles or characterizations of   isoperimetric regions, see for example, \cite{Morgan2}, \cite{Morgan3}, \cite{Petean0} and \cite{Petean3}. Nevertheless, for  many seemingly  simple products like $(S^2\times \re^2,g_0+g_E)$ or $(S^1\times S^1 \times \re^2,dt^2+ds^2+g_E)$, the explicit isoperimetric profiles are not known.

Let $(T^2,h_2)$ be a standard flat torus, $T^2=\re^2/\Gamma$, with $\Gamma$ the orthogonal lattice generated by  $\{(2 \pi r_1,0),(0,2 \pi r_2)\}$, $r_1,r_2 >0$. For example, the precise isoperimetric profile of  $(T^2\times \re ,h_2+dt^2)$ is not known, and is conjectured to be a profile generated by regions  such as spheres ($B^{3}_R$), cylinders ($B^{2}_R\times [0,r]$) and planes ($T^2\times [0,r]$), $R,r>0$. This conjecture was proven to be true for small volumes $0<v<v_1^*$, for some $v_1^*>0$ by the work of  L. Hauswirth, J. P\'erez and A. Ros, in \cite{Hauswirth}. Moreover, one may notice that the conjecture is also true for big volumes $v>v_1^{**}$, for some $v_1^{**}>0$, through an immediate application of the Ros product Theorem (\cite{Ros}, Theorem 3.7), and a comparison with the isoperimetric profile of $S^2\times \re$, computed by R. Pedrosa \cite{Pedrosa}.  More precisely, let $(T^2,h)$ be the flat torus with lattice generated by $\{(2\sqrt{ \pi},0),(0,2\sqrt{\pi})\}$. By direct computation
$V^2(T^2)=4\pi=V^2(S^2)$ and  $I_{S^2}\leq I_{T^2}$. Being $S^2$ a model metric, we may apply the Ros product Theorem to the inequality. This yields $I_{S^2\times \re}\leq I_{T^2\times \re}$.

  Now, let $B^{n}_R\subset \re^n$ denote a ball of radius $R>0$, and let $f_n:[0,\infty)\rightarrow [0,\infty)$ be the function given by 
	
	\begin{equation}
		\label{fv}
		f_n(v)=V^{n+1}(\partial(T^2\times B_R^{n})),
\end{equation}

\noindent  with $R$ such that  $v=V^{n+2}(T^2\times B^{n}_R)$. With this notation, since $T^2\times B^{1}_R$ are actual closed regions in $T^2\times \re$, one has

$$I_{S^2\times \re}(v)\leq I_{T^2\times \re}(v)\leq f_1(v).$$
Explicit computations of $I_{S^2\times \re}$ in the before cited work of R. Pedrosa \cite{Pedrosa}, show that for $v\geq v^{**}$, $I_{S^2\times \re}(v)= f_1(v)$, with $v^{**}\approx 16.66$.
It follows that $I_{T^2\times \re}(v)= f_1(v)$ for $v\geq v^{**}$.

 We may resume the above discussion in the following Theorem. 

\begin{Theorem} \normalfont{(Theorem 18 in \cite{Hauswirth}, together with \cite{Pedrosa} and  \cite{Ros})}
	
		\label{T2xR} 
	\emph{	Let $(T^2,h)$ be a standard flat torus, $T^2=\re^2/\Gamma$, with $\Gamma$ the orthogonal lattice generated by   $\{(2\sqrt{ \pi},0),(0,2\sqrt{\pi})\}$.
There are some $v_1^*, v_1^{**}>0$ such that the isoperimetric profile of $(T^2\times \re,h+dr^2)$ satisfies the following.	For $v<v_1^*$, $I_{T^2\times \re}(v)= I_{\re^3}(v)$ and for   $v>v_1^{**}$, $I_{T^2\times \re^n}(v)= f_1(v).$
	Explicit estimates are  $v_1^*=\frac{32 \pi ^{5/2}}{81}\approx 6.91$ and $v_1^{**}\approx16.66$.}
	
\end{Theorem}

In this article we paint a similar picture for the Riemannian manifold $(T^2\times \re^n,h_2+g_E)$, for $2\leq n\leq 5$ and $g_E$ the Euclidean metric on $\re^n$. Our first result is the following.

\begin{Theorem}
	\label{T2}
		Let $(T^2,h_2)$ be a standard flat torus, $T^2=\re^2/\Gamma$, with $\Gamma$ the orthogonal lattice generated by   $\{(0, 2\pi r_1),( 2\pi r_2,0)\}$, $0<r_1\leq r_2$.
For $2\leq n\leq 5$, there are some $\tilde v_n^*$ and $\tilde v_n^{**}$ such that the isoperimetric profile of $(T^2\times \re^n,h_2+g_E)$ satisfies the following.	For $v\leq\tilde v_n^*$, $I_{T^2\times \re^n}(v)= I_{S^1\times \re^{n+1}}(v)$ and for   $v\geq \tilde v_n^{**}$, $I_{T^2\times \re^n}(v)= f_n(v).$

	\end{Theorem}

  Moreover, our proof  gives simple formulas to compute explicit lower bounds for   $\tilde v_n^*$ and upper bounds for $\tilde v_n^{**}$, as functions only of $n, r_1, r_2$. For example, a lower bound for  $\tilde v_n^{**}$ is  $v_n^{**}= \max\{a_n, b_{n}\}$, where  $a_n$ is such that 
 $I_{S^1_{r_1}\times \re^{n+1}}(a_n)- f_n(a_n) = 2\beta_n(r_1)$, and  
 $b_n$ such that
$ I_{S^1_{r_2}\times \re^{n+1}}(b_n)- f_n(b_n) = 2\beta_n(r_2)$. $\beta_n(r)$ being given by eq. (\ref{beta}).
 
 On the other hand, an upper bound for $\tilde v_n^*$ is  $v_n^*=\min\{c_n, v_s\}$, where $v_s = \min \{V^{n+2} (S^1_{r_1} \times B^{n+1}_{\pi r_2}), V^{n+2} (B^{n+2}_{\pi r_1}) \}$ and $c_n$ is such that 
 $I_{S^1_{r_1}\times S^1_{r_2}\times \re^{n+1}}(c_n) =K^*$, where $K^*>0$ is given by Lemma 
 \ref{tm}. See the proof of Theorem \ref{T2} for details on these estimates. Numerical estimates  for  $v_2^*\leq \tilde v_2^* $ and $v_2^{**}\geq \tilde v_2^*$ for $r_1=r_2=1$ are $v_2^{*}\approx 5.25$, and $v_2^{**}\approx 70.12 $. 
  
  \begin{figure}
  	\includegraphics[scale=.3]{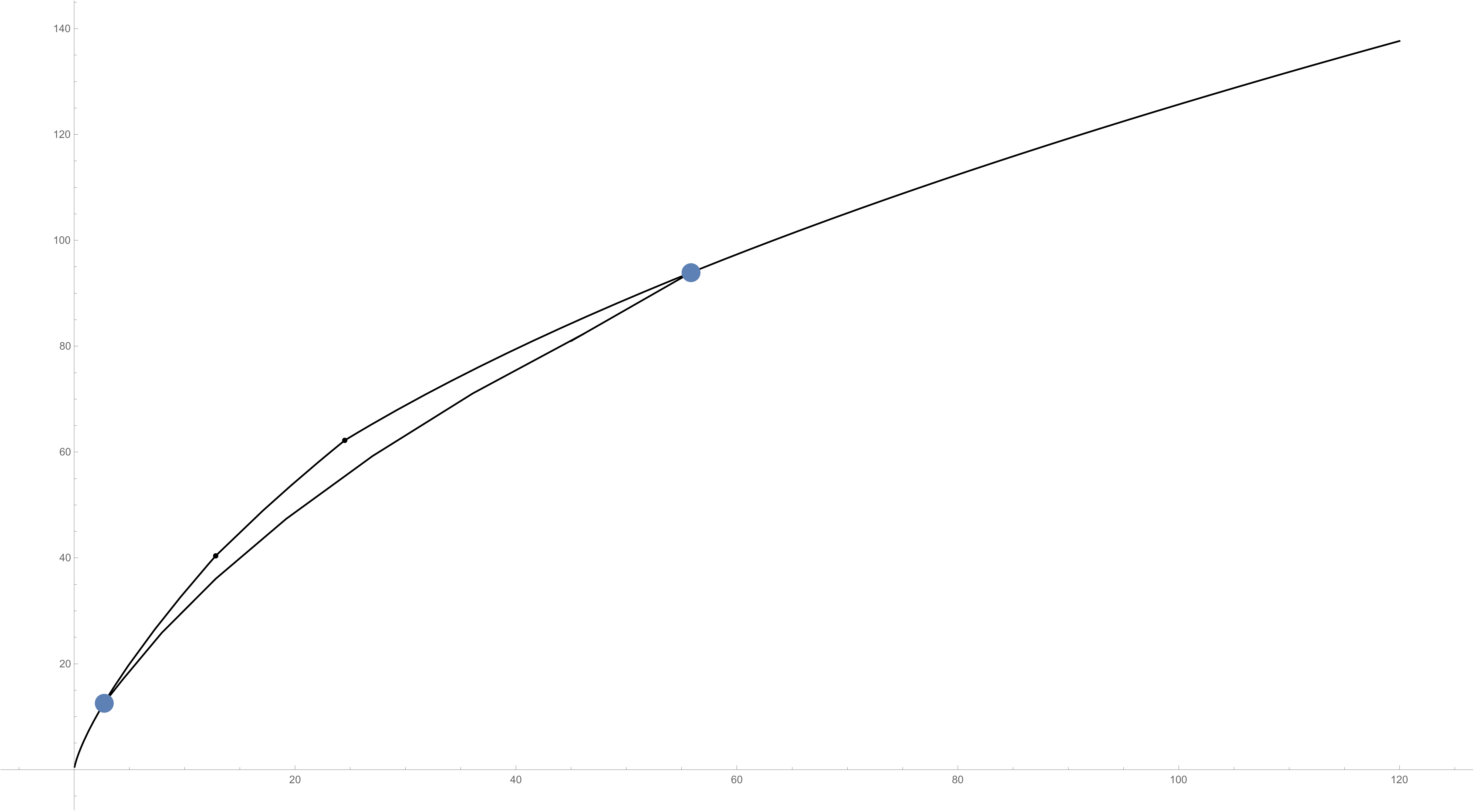}
  	\caption{ Before $v_2^*\approx 2.7$ and after $v_2^{**}\approx 55.8$, the isoperimetric profile of  $(T^2\times \re^2,h+g_E)$ is known precisely. In the interval between $v_2^*$ and  $v_2^{**} $, is bounded above and below. See example \ref{example} for details. $(T^2,h)$ is a standard flat torus, $T^2=\re^2/\Gamma$, with $\Gamma$ the orthogonal lattice generated by   $\{(2\sqrt{ \pi},0),(0,2\sqrt{\pi})\}$. $g_E$ is the Euclidean metric.}
  	  	\label{isopprofT2R2}
  \end{figure}

The bounds for  $\tilde v_n^*$ and $\tilde v_n^{**}$ we give are not optimal. In fact, we conjecture that $\tilde v_n^*=\tilde v_n^{**}$. That is, $I_{S_{r_1}^1\times S_{r_2}^1\times \re^{n}}(v)=I_{scp}(v)$, where $I_{scp}(v)=\min\{I_{S_{r_1}^1\times \re^{n+1}}(v), f_n(v)\}$.

Through the same symmetrization techniques one can obtain corresponding results for $T^3\times \re^n$, based on the estimates for  the isoperimetric profile of $T^{2}\times \re^n$. We first define a corresponding function for the area  of  regions of the type $V^{n+3}(\partial(T^3\times B_R^{n}))$. Given a volume $v$ consider the function 
$g_n(v)=V^{n+2}(\partial(T^3\times B_R^{n}))$,
with $R$ such that  $V^{n+3}(T^3\times B^{n}_R)=v$.

\begin{Theorem}
	\label{T3}
Let $(T^3,h_3)$ be a standard flat k-torus, $T^3=S^1_{r_1}\times S^1_{r_2}\times S^1_{r_3} $, $r_1\leq r_2\leq r_3$. Let $2\leq n\leq 4$.
Suppose that  there are some $v_{n}^*, v_{n}^{**}$, with $v_{n}^{**} \geq v_{n}^*>0$, such that the following is satisfied.	For $v\leq v_{n}^*$, $I_{S^1_{r_1}\times S^1_{r_2}\times \re^n}(v)= I_{ S^1_{r_1}\times \re^{n+1}}(v)$. For   $v\geq v_{n}^{**}$,  $I_{S^1_{r_1}\times S^1_{r_2}\times \re^n}(v)=  f_{n}(v).$

Then, there are some $\tilde u_{n}^*, \tilde u_n^{**}$, with  $\tilde u_{n}^{**}\geq \tilde u_{n}^*>0$, such that  the following is satisfied.	For  $v\leq \tilde u_n^*$,  $I_{S^1_{r_1}\times \re^{n+2}}(v)= I_{S^1_{r_1}\times S^1_{r_2}\times S^1_{r_3} \times \re^{n}}(v)$. And for   $v\geq\tilde u_{n}^{**}$,  $I_{S^1_{r_1}\times S^1_{r_2}\times S^1_{r_3} \times \re^n}(v)=  g_{n}(v).$
\end{Theorem}
	
The proof of Theorem \ref{T3} gives explicit estimates for lower bounds for  $\tilde u_n^*$ and upper bounds for $\tilde u_{n}^{**}$, based on those for  $v_{n}^{*}$ and $v_{n}^{**}$.  Of course, one may combine the ideas behind the proofs of Theorems \ref{T2} and \ref{T3} in an inductive way and  obtain corresponding results for $T^k\times \re^n$, $2\leq k\leq5$, $2\leq n\leq 7-k$, based on the estimates for  $(T^{k-1}\times \re^n,h+g_E)$.%; getting estimates in each step of the process, for the bounds for $\tilde u_n^*$ and $\tilde u_{n}^{**}$.

%More general estimates for $T^k\times \re^n$, $n\leq 7-k$, may be derived using these techinques in an inductive way.

%The isoperimetric profile for $S^1\times \re^{n}$, $n\leq 6$, has been described by R. Pedrosa and M. Ritor\'e, \cite{Pedrosa}. isoperimetric regions in $S^1\times \re^{n+1}$, are either round balls or tubes of the type $S^1 \times B^{n+1}_R$. Theorem \ref{IP} then implies that for $(T^2\times \re^n,h+g_E)$, $1\leq n\leq 5$, isoperimetric regions are either round balls $B^{n+2}$ or cylinders of the type $S^1 \times B^{n+1}_R$ or products of the type $T^2\times B_R^{n}$. 
The fact that, for big volumes, regions of the type  $M\times B_R^n$ are isoperimetric regions  of manifolds the type $M^k\times \re^n$, where $M^k$ is compact, was known to be true in general (see for example the work of J. Gonzalo \cite{Gonzalo}). Nevertheless, no explicit estimates of how big the volume should be, in order for this to happen, were known. On the other hand, isoperimetric regions with small volumes  were studied in the case $T^2\times \re$, in \cite{Hauswirth}, using symmetries and properties exclusive of three manifolds. Our approach is   different,  based on symmetrization techniques like the Ros product Theorem \cite{Ros} and  others introduced by F. Morgan in \cite{Morgan}. We also treat the more general case $T^2\times \re^n$.

Estimates for $v_n^*$ and $v_n^{**}$ give a good understanding of the general shape of the isoperimetric profile of $(T^2\times \re^n, h_2+g_E)$. For example, figure \ref{isopprofT2R2} shows lower and upper bounds for  the graphic of the isoperimetric profile of $(T^2\times \re^2,h+g_E)$; based on computations of $v_2^*$ and $v_2^{**}$.
\\

\textbf{Acknowledgments.}
The authors were supported by grant UNAM-DGAPA-PAPIIT IA106918. The authors would like to thank Professor Adolfo S\'anchez Valenzuela and CIMAT M\'erida for their hospitality, where part of this work was done. We would also like to thank Professor Mario Eudave Mu\~noz from IMATE-UNAM Juriquilla, for useful comments on the subject.

\section{Notation and background}

Existence and regularity of isoperimetric regions is a fundamental result due to the works of Almgren \cite{Almgren}, Gr\"uter \cite{Gruter}, Gonzalez, Massari, Tamanini \cite{Gonzalez}, (see also Morgan \cite{MorganBook}, Ros \cite{Ros}).

\begin{Theorem}
	\label{existence}
 Let $M^n$ be a compact Riemannian manifold, or non-compact with $M/G$ compact, being $G$ the
 isometry group of $M$. Then, for any $t$, $0 < t < V(M)$, there exists a
compact region $\Omega \subset M$, whose boundary $\Sigma = \partial \Omega$ minimizes area among regions of
volume $t$. Moreover, except for a closed singular set of Hausdorff dimension at most
$n - 8$, the boundary $\Sigma$ of any minimizing region is a smooth embedded hypersurface
with constant mean curvature.

\end{Theorem}

Note that $T^2\times \re^n$ has no boundary, and is compact if it is acted upon by its isometry group. Also since we will only be dealing with the cases $n+2\leq 7$, every hypersurface $\Sigma$ enclosing an isoperimetric region will be smooth and of constant mean curvature (CMC).
\begin{comment}
A closed constant mean curvature hypersurface is stable if it minimizes area up to second order for fixed volume. Isoperimetric regions are CMC stable hyperurfaces (cite). A more detalied definition of stability will be given in section  

\begin{Remark}
	\label{sigmaconnected}
	It is well known that in the case $Ric\geq 0$, stable CMC hypersurfaces are connected or totally geodesic. 
	(see \cite{Ros}, p. 16)
\end{Remark}
\end{comment}
Throughout the article, $B^n_{R}$ will denote an n-dimensional ball of radius $R$ in $\re^n$.

We will work only with Tori that are Riemannian products: $(T^2,g_2)= (S^1_{r_1}\times   S^1_{r_2} , ds^2+dt^2)$ or $(T^3,g_3)= (S^1_{r_1}\times   S^1_{r_2}\times   S^1_{r_3} , ds^2+dt^2+du^2)$, for some $r_1, r_2, r_3 \in \re^+$. 
Without loss of generality we will assume $r_1\leq r_2 \leq r_3$.

The isoperimetric profile of the Riemannian product $S^1_r\times \re^n$, $2\leq n\leq7$  is well known, by the work of Pedrosa and Ritor\'e (Theorem 3.5 in \cite{Pedrosa})  and is given by

\begin{equation}
\label{Is1}
I_{S^1_r\times \re^n}(v) =
\begin{cases}
(1+n)^{\frac{n}{1+n}}\omega_n^{\frac{1}{1+n}} v^{\frac{n}{1+n}}, & \text{if}\ v\leq \beta_n(r) \\
n^{\frac{n-1}{n}}(2\pi r \omega_{n-1})^{\frac{1}{n}} v^{\frac{n-1}{n}}, & \text{if}\  v > \beta_n(r) \\
\end{cases}
\end{equation}

\noindent where $\omega_n=V^n(S^n)$ and  

\begin{equation}
\label{beta}
\beta_n(r)=n^{(n-1)(n+1)}(2 \pi r \omega_{n-1})^{n+1} (1+n)^{-n^2} \omega_n^{-n}.
\end{equation}
\noindent  Note that $\beta_n(r)$ depends only on $n$ and $r$. 
 For fixed $r$ and $n$, $\beta_n(r)$ is the critical number such that for volumes less than  $\beta_n(r)$, balls $B^{n+1}_R\subset S^1_r\times \re^n$ are isoperimetric; while for volumes greater than  $\beta_n(r)$, regions of the type $S_r^1\times B^n_R \subset S^1\times \re^n$ are isoperimetric.
The isoperimetric profile is continuous. We will denote by $\alpha_n(r)$ the area of the isoperimetric regions of volume $\beta_n(r)$; this is,  $\alpha_n(r)=I_{S^1_r\times \re^n}(\beta_n(r))$.

\section{The isoperimetric profile of $T^2\times \re^n$}

It was conjectured in \cite{Hauswirth} that the isoperimetric profile of $T^2\times \re$ is composed of three parts:  for small volumes, the solutions of the isoperimetric problem are  spheres $(B^3_R)$, then, for intermediate volumes, cylinders ($S^1\times B^2_R$), then, for big volumes, planes ($T^2 \times  B^1_R$). This was called the $I_{scp}$ profile (spheres-cylinders-planes). The conjecture is then that $I_{T^2 \times \re}=I_{scp}$. In the same article, the conjecture was proven to be true for small volumes. The solutions were also proved to be unique and their proof included tori of other types, more general than only orthogonal tori.

We propose a similar conjecture for   $T^2\times \re^n$: for small volumes,  spheres $B^{n+2}_R$ are isoperimetric regions, for intermediate volumes, cylinders ($S^1\times B^{n+1}_R$), and  for big volumes, planes  ($T^2 \times  B^n_R$).
We will also call this the $I_{scp}$ profile. We will not discuss uniqueness of solutions to the isoperimetric problem. Our results make use of equation (\ref{Is1}), so that in the following, $n$ is an integer such that $2\leq n\leq 7$.

Let $\Omega$ be an isoperimetric region in $T^2\times \re^n=S^1_{r_1}\times S^1_{r_2}\times \re^n$. We may parameterize $S^1_{r_1}$ by $[0,2 \pi r_1]$ and consider the slices $\Omega_t$, $t\in [0,2\pi \  r_1),$

$$\Omega_t=\Omega \cap( \{t\}\times S^1_{r_2}\times \re^n).$$

\noindent Then for each slice $\Omega_t$,   we may compute its $(n+1)$-volume and define a function $F_{1}:[0,2 \pi r_1]\rightarrow \re$, by $F_{1}(t)=V^{n+1}(\Omega_t)$ and $F_{1}(2 \pi r_1)=V^{n+1}(\Omega_0)$.

Similarly, one may parameterize $S^1_{r_2}$ by $[0,2 \pi r_2)$ and consider the slices $\Omega_s$, $s\in [0,2\pi  r_2)$:
$$\Omega_s=\Omega \cap(  S^1_{r_1}\times \{s\}\times \re^n).$$

 Likewise  we may define   $F_{2}:[0,2 \pi \ r_2]\rightarrow \re$, by $F_{2}(s)=V^{n+1}(\Omega_s)$ and $F_{2}(2 \pi r_2)=V^{n+1}(\Omega_0)$. Of course, both $F_{1}$ and $F_{2}$ are continuous.
 Let $\theta_m$ and $\theta_M$,  and $\sigma_m$ and $\sigma_M$, denote the minimum and maximum values of $F_{1}(t)$  and of $F_{2}(s)$ respectively. 	

We start with the following.

\begin{Lemma}
	\label{vol0}
	If  $\theta_m =0$ or $\sigma_m=0$,  then $$ I_{S_{r_1}^1\times \re^{n+1}}(V^{n+2}(\Omega)) \leq V^{n+1}(\partial \Omega).$$ 
	%Moreover, if also $v < v_0$, then 
	%$$I_{S_{r_1}^1\times \re^{n+1}}(v) = I_{S_{r_1}^1\times{S_{r_2}^1 \times \re^{n+1}}} (v)$$
	%with $v_0 = V^{n+2}(S_{r_1}^1\times B^{n+1}_{\pi r_1})$.
\end{Lemma}
\begin{proof}
Suppose	$\theta_m=0$.  Let $t_0 \in [0, 2 \pi r_1]$ be such that $F_1(t_0)=\theta_m=0$.
%Then $\Omega_{t_0}$ has  zero $(n+1)$-volume in $\{t_0\}\times S_{r_2}^1\times \re^n$. 
We construct a new closed region $\Omega^* \subset [t_0,t_0+2 \pi r_1] \times S_{r_2}^1 \times\re^{n}$, in the following way.
 For $t\in [0, 2\pi r_1)$, let $\Omega^*_{t_0+t}=\Omega_t$. Also, let $\Omega^*_{t_0+2\pi r_1}=\Omega_{t_0}$. That is, we are adding a copy of $\Omega_{t_0}$ at $\{t_0+2 \pi r_1\}\times S_{r_2}^1\times \re^n$.  Since, by hypothesis  $V^{n+1}(\Omega_{t_0})=0$, we then have  $V^{n+1}(\partial \Omega)=V^{n+1}(\partial\Omega^*)$ and $V^{n+1}(\Omega)=V^{n+1}(\Omega^*)$. Note also that $\Omega^*$ is a closed region in $\re\times S^1_{r_2} \times \re^n$, by continuity of $F_1$. It follows that 
\begin{equation}
\label{firstfirst}
I_{S_{r_2}^1\times \re^{n+1}}(V^{n+2}(\Omega))=I_{S_{r_2}^1\times \re^{n+1}}(V^{n+2}(\Omega^*))\leq V^{n+1}(\partial \Omega^*) = V^{n+1}(\partial\Omega).
\end{equation}
 
Finally, since $r_1\leq r_2$,  eqs.  (\ref{Is1}) and (\ref{firstfirst})  imply 
$$I_{S_{r_1}^1\times \re^{n+1}}(V^{n+2}(\Omega))\leq I_{S_{r_2}^1\times \re^{n+1}}(V^{n+2}(\Omega)) \leq V^{n+1}(\partial\Omega).$$  

The proof of the case	$F_2(s_0)=0$ is very similar, as in this case $\Omega$ can also be embedded isometrically in $S_{r_1}^1\times \re \times\re^{n}$ as a closed region, by adding an $(n+1)$ zero measure set $\Omega_{s_0}$.  Hence  in this case we also have $I_{S_{r_1}^1\times \re^{n+1}}(V^{n+2}(\Omega))\leq V^{n+1}(\partial\Omega)$. %The conclusion of the lemma follows.
  
%Moreover
\end{proof}

%Recall that $\beta_n(r_2)$ is defined in eq. (\ref{Is1}), for $I_{S_{r_2}^1\times \re^{n+1}}(v)$, as the critical volume $v=\beta_n(r_2)$ before which, balls $B^{n+1}_R$ are isoperimetric regions and after which, regions of the type $S^1_{r_2}\times B^n_{R}$ are isoperimetric.

\begin{Lemma}
	\label{vol1}
	If $\theta_M \leq \beta_n(r_2)$ or  $\sigma_M \leq \beta_n(r_1)$, then $ I_{S^1_{r_1}\times \re^{n+1}}(V^{n+2}(\Omega)) \leq V^{n+1}(\partial \Omega)$.
\end{Lemma}
\begin{proof}
	We will suppose  $\theta_M \leq \beta_n(r_2)$; the proof of the case $\sigma_M \leq \beta_n(r_1)$ is   similar. We will also suppose $\theta_m >0$ and $\sigma_m>0$ since the case $\theta_m =0$ or $\sigma_m=0$ is treated in Lemma \ref{vol0}.
	
 The idea is to  symmetrize the isoperimetric region $\Omega$ as in the proof of the Ros Product Theorem (\cite{Ros}). We construct a new region $\Omega^* \subset S^1_{r_1}\times S^1_{r_2}\times \re^n$ by replacing each $\Omega_t\subset\{t\}\times S_{r_2}^1\times \re^n$ with an isoperimetric region in $\{t\}\times S_{r_2}^1\times \re^n$. That is, we let $\Omega^*_t=\{t\}\times B_{R(t)}^{n+1}$, where $R(t)>0$ is such that  $V^{n+1}(\{t\}\times B_{R(t)}^{n+1})=V^{n+1}(\Omega_t)$.  Since $\theta_M\leq \beta_n(r_2)$, then $V^{n+1}(\Omega_t)\leq \beta_n(r_2)$ for each $t\in [0,2\pi r_1)$ and hence $\{t\}\times B_{R(t)}^{n+1}\subset  \{t\}\times S_{r_2}^1\times \re^n$ for each $t$.
 
 Also, since $F_{1}(t)$ is continuous, the region $\Omega^*$ is closed in $S^1_{r_1}\times S^1_{r_2}\times \re^n$. Note also that by construction $V^{n+2}(\Omega^*)=V^{n+2}(\Omega)$.
	
	Recall  from eq. (\ref{Is1}) that for $v\leq \beta_n(r_2)$, $ I_{S^1_{r_2}\times \re^{n}}(v)=I_{\re^{n+1}}(v)$. Since  $\theta_M \leq \beta_n(r_2)$, it follows that  for each $t\in [0,2\pi r_1)$:
	
	$$V^{n+1}(\partial\Omega^*_t)=V^{n+1}(\partial B_{R(t)}^{n+1})=I_{\re^{n+1}}(V^{n+2}(\Omega_t))= I_{S^1_{r_2}\times \re^{n}}(V^{n+2}(\Omega_t)) \leq V^{n+1}(\partial \Omega_t).$$
	
  \noindent Arguing as in the proof of the Ros Product  Theorem, from the last inequality we get
	$$V^{n+1}(\partial \Omega^*)\leq V^{n+1}( \partial \Omega).$$
	Moreover, since $V^{n+2}(\Omega^*)=V^{n+2}(\Omega)$ and  $\Omega^*$ is a closed region in $S^1_{r_2}\times \re^{n+1}$, we have 	
	 $$I_{S^1_{r_2}\times \re^{n+1}}(V^{n+2}(\Omega))=I_{S^1_{r_2}\times \re^{n+1}}(V^{n+2}(\Omega^*))\leq  V^{n+1}(\partial \Omega^*) \leq V^{n+1} (\partial \Omega).$$
	
	Since $r_1\leq r_2$ the conclusion of the lemma follows.
	
\end{proof}

Recall the definition of $f_n(v)$ by equation (\ref{fv}). We prove the following.
%: given a volume $v$, $f_n(v)$ gives us the area of the product $T^2\times B^n_R$, with $R$ being such that the volume of the region  $T^2\times B^n_R$  is $v$.

\begin{Lemma}
	\label{vol2}
	If $\theta_m \geq \beta_n(r_2)$ or $\sigma_m \geq \beta_n(r_1)$, then 
	$$ V^{n+1}( \partial \Omega)=f_n(V^{n+2}(\Omega)) .$$ 
\end{Lemma}
\begin{proof}
	
	We will prove the case $\theta_m \geq \beta_n(r_2)$. The other one is similar.
	
	Recall from eq. (\ref{Is1}) that for  $v\geq \beta_n(r_2)$, isoperimetric regions in $S^1_{r_2}\times \re^n$ are of the type $S^1_{r_2}\times B^n_R$. This means $I_{S^1_{r_2}\times \re^{n}}(v)=V^1(S^1_{r_2}) V^{n-1} (\partial B^n_R)$, for some $R>0$ such that $v=V^1(S^1_{r_2}) V^n(B^n_R)$. 
We symmetrize the isoperimetric region $\Omega$ as in the proof of the Ros Product Theorem: we replace each $\Omega_t$  in $\{t\}\times S^1_{r_2}\times \re^n$ by a product of $S^1_{r_2}$ and ball $B_{R(t)}^{n}\subset \re^n$ such that  $V^1(S^1_{r_2}) V^n(B_{R(t)}^{n})=V^{n+1}(\Omega_t)$. We denote the new region in $ S^1_{r_1}\times S^1_{r_2}\times \re^n$ by $\Omega^*$. Since $F_{1}(t)$ is continuous, and for each $t$ we are using a region of the type $  S^1_{r_2}\times B_{R(t)}^{n}$, the region $\Omega^*$ is closed in $ S^1_{r_1}\times S^1_{r_2}\times \re^n$. Note also that $V^{n+2} (\Omega^*)=V^{n+2} (\Omega)$. And, since for each $t \in [0,2\pi r_1)$ we have 
$$ V^{n} (\partial \Omega^*_t)=I_{S^1_{r_2}\times \re^{n}}(V^{n+1} (\Omega_t))\leq V^{n} (\partial \Omega_t),$$ 
it follows from the Ros product Theorem that 
\begin{equation}
\label{firstsym}
V^{n+1} (\partial\Omega^*)\leq V^{n} (\partial\Omega).
\end{equation}

We now symmetrize $\Omega^*\subset S^1_{r_1}\times S^1_{r_2}\times \re^n$ with respect to the other factor, $S^1_{r_2}$. We parameterize  $S^1_{r_2}$ by $[0,2 \pi \ r_2)$ and consider the slices $\Omega^*_s$, $s\in [0,2\pi r_2)$:

$$\Omega^*_s=\Omega^* \cap(  S^1_{r_1}\times \{s\}\times \re^n).$$

For each slice we may compute its $(n+1)$-volume and define a function $G:[0,2 \pi r_2]\rightarrow \re$, by
$G(s)=V^{n+1}(\Omega^*_s)$ for $[0,2 \pi r_2)$  and $G(2\pi r_2)=V^{n+1}(\Omega^*_0)$. Note that $G$ is continuous. Moreover,   by construction,  for each $t$ and any  $s_1, s_2 \in [0,2 \pi r_2]$,

$$\{t\}\times\{s_1\}\times  B^n_{R(t)}=\{t\}\times\{s_2\}\times  B^n_{R(t)}.$$

\noindent This implies that both slices $\Omega^*_{s_1}$ and $\Omega^*_{s_2}$ have the same volume.   It follows that $G(s)$ is constant.

%consider  and the slices $\Omega^*_{s_1}$ and $\Omega^*_{s_2}$. N

 % Note  that after the first symmetrization we may parameterize $\Omega^*$ with $\{t\}\times \{s\}\times B^n_{(R_t)}$, where $t\in [0, 2 \pi \ r_1]$ and  $s\in [0, 2 \pi \ r_2]$. Hence $\Omega^*_{s_1,t}=\{t\}\times \{s_1\}\times$  ,  by the first symmetrization, for each $t\in [0, \pi r_1)$. We conclude  $G(s)$ is constant.
We now claim that $G(s)\geq \beta_n(r_2)$: by hypothesis $\theta_m \geq \beta_n(r_2)$, which implies 
\begin{equation}
\label{lessthanbeta}
V^{n+2}(\Omega)\geq V^1(S_{r_1})\beta_n(r_2).
\end{equation}

\noindent  If the claim were not true, then $G(s)< \beta_n(r_2)$ and we would have
$$V^{n+2}(\Omega^*)=V^1(S_{r_1})G(s)<V^1(S_{r_1})\beta_n(r_2),$$

\noindent which is ruled out by inequality (\ref{lessthanbeta}), since  $V^{n+2} (\Omega^*)=V^{n+2} (\Omega).$
%, we may apply Lemma \ref{vol1} to $\Omega^*$ and obtain 
 
%\begin{equation}
%\label{f1} 
%I_{S^1_{r_1}\times \re^{n+1}}(V^{n+2}(\Omega)) \leq V^{n+1}(\partial \Omega).
%\end{equation}
We now construct a new region $\Omega^{**}\subset S^1_{r_1}\times S^1_{r_2} \times \re^n$ by letting each slice $\Omega^{**}_{s}= S^1_{r_1}\times \{s\} \times B_{R_0}^{n}$, where $R_0$ is such that  $V^1(S^1_{r_1}) V^n(B_{R_0}^{n})=G(s)$.

Since  $G(s)$  is constant, then $R_0$ is constant, and we get $V^{n+2}(\Omega^{**})=V^{n+2}(\Omega^*)=V^{n+2}(\Omega)$. Moreover, by continuity of $G(s)$, the region $\Omega^{**}$ is closed. And  since $G(s)\geq \beta_n(r_2)$, we have 
$$V^{n}(\partial\Omega^{**}_s)= V^{n}(\partial(S^1_{r_1}\times \{s\} \times B_{R_0}^{n})) \leq V^{n}(\partial\Omega^{*}_s).$$
\noindent Hence,   using the Ros Product Theorem we get $V^{n+1}(\partial \Omega^{**})\leq V^{n+1}(\partial\Omega^*)$. And together with eq. (\ref{firstsym}):

$$V^{n+1}(\partial \Omega^{**})\leq V^{n+1}(\partial\Omega^*)\leq V^{n+1}(\partial\Omega).$$

Finally, by construction, we have that in fact
 $$\Omega^{**}= S^1_{r_1}\times S^1_{r_2} \times B_{R_0}^{n}.$$
 
 It  follows that 
 $$f_n(V^{n+2}(\Omega))= V^{n+1}(\partial(S^1_{r_1}\times S^1_{r_2} \times B_{R_0}^{n}))= V^{n+1}(\partial \Omega^{**})\leq V^{n+1}(\partial\Omega).$$

Being $\Omega$ isoperimetric, we conclude that $f_n(V^{n+2}(\Omega))= V^{n+1}(\partial\Omega)$.

%and the conclu Lemma follows.

\end{proof}

%We are now ready to prove the first part of Theorem \ref{T2}.
We now prove that the case $0<\theta_m<\beta_n(r_2)<\theta_M$ and $0<\sigma_m<\beta_n(r_1)<\sigma_M$ cannot occur for small areas of $\Omega$.

\begin{Lemma}
	\label{tm} Suppose that $0<\theta_m<\beta_n(r_2)<\theta_M$ and $0<\sigma_m<\beta_n(r_1)<\sigma_M$ occurs.

 Then there is some $K^*>0$ such that  $V^{n+1}(\partial \Omega) \geq K^*.$ Moreover, $K^*$  is independent of $
  \Omega$ and depends only on $r_1,r_2,n$. In particular, is given by
\begin{equation}
\label{K}
K^*=\max\{V^1(S^1_{r_1}) \ I_{  \re^{n+1}} (\theta^*), V^1(S^1_{r_2}) \ I_{  \re^{n+1}} (\sigma^*) \},
\end{equation}
  \noindent where $\theta^*\in (0,\beta_n(r_1))$ is such that
  \begin{equation}
  \label{alg1}
  \frac{1}{2} V(S^1_{r_2})I_{\re^{n+1}}(\theta^*)+ \theta^* = \beta_n(r_1),
  \end{equation}
  
  \noindent and $\sigma^*\in (0,\beta_n(r_2))$ such that
  \begin{equation}
  \label{alg2}
  \frac{1}{2} V(S^1_{r_1})I_{\re^{n+1}}(\sigma^*)+\sigma^* =\beta_n(r_2).
  \end{equation}
  
% Moreover, for $v < v^*$
 %$$I_{S_{r_1}^1\times \re^{n+1}}(v) = 	I_{S^1_{r_1}\times S^1_{r_2} \times \re^n}(v),$$
%where 
%$v^*$ is such that $I_{S_{r_1}^1\times \re^{n+1}}(v^*) = K^*$.
	
%\begin{equation}
%\label{Prop1}
%	I_{S^1_{r_1}\times S^1_{r_2} \times \re^n}(v)= \min\{ I_{S_{r_1}^1\times \re^{n+1}}(v),f_n(v) \}
%\end{equation}
\end{Lemma}

\begin{Remark}
	\label{solution}
	 Since $r_1$ and $r_2$ are fixed, equations (\ref{alg1}) and (\ref{alg2}) are algebraic equations   of the type $a \  x^{\frac{n+1}{n+2}}+x=b$, with $a,b,n>0$. It is straightforward to check that a solution exists and is  unique for each equation. %Let $\varphi(x)=a \  x^{\frac{n+1}{n+2}}+x$. Since $a,b>0$, we have $\varphi(0)<b$ and $\varphi(b)>b$; by continuity of  $\varphi(x)$, a solution $\varphi(x)=b$ exists in $(0,b)$. Uniqueness follows from the fact that $\varphi(x)$ is  increasing.
 \end{Remark}

\begin{proof}

%Hence, for the proof, we will only check the case  that $0<t_m<\beta_n(r_2)<t_M$ and $0<s_m<\beta_n(r_1)<s_M$ occur.

We follow a construction by F. Morgan \cite{Morgan}, which estimates lower bounds for isoperimetric profiles of products. We consider the product of $(S^1_{r_1},dt^2)$ with $(S^1_{r_2}\times \re^n,ds^2+g_E)$. We start by defining a product manifold  $(0,V_1)\times (0,\infty)\subset \re^2$, where $V_1=V^1(S^1_{r_1})$. And we equip this 2-dimensional manifold with a model metric in the sense of the Ros product Theorem (\cite{Ros}).   $(0,V_1)$  and $(0,\infty)$ will have  Euclidean Lebesgue Measure and Riemannian metric $\frac{1}{2}ds$ and $(\frac{1}{h(x)})dr$ respectively, where $h(x)= I_{S_{r_2}^1\times \re^{n+1}}(x)$. 

To show that this is in fact a model metric, it suffices to prove that in each interval, $(0, V_1)$ and $(0, \infty)$, intervals of the type $(0,t)$, $t>0$, minimize perimeter, among closed sets $S$ of given Euclidean length $t$. For the interval $(0, \infty)$ this holds because $h(x)$ is nondecreasing. On the other hand, for the interval $(0,V_1)$, we argue as follows. Suppose $S\subset(0,V_1)$ is a closed set of perimeter $t$, that is not of the type $(0,t)$; then it must be a locally finite collection of closed intervals; then an interior interval must be at least borderline unstable,  because   the factor $\frac{1}{2}$ is constant. We conclude that $S$ does not minimize perimeter.  %$\frac{1}{2}ds$.

Minkowski content on $(0,V_1)$  and $(0,\infty)$ counts boundary points of intervals with density $2$ and $h(x)$, respectively. Similarly, Minkowski content on 
$(0,V_1)\times (0,\infty)$ has perimeter measured by 

\begin{equation}
\label{Morg}
ds^2= h^2(v_2) dv_1^2+ 2^2 dv_2^2.
\end{equation}

It follows from the proof of the Ros Product Theorem that, for any $v>0$,  $I_{S_{r_1}^1\times (S_{r_2}^1 \times \re^n)}(v)$ is bounded from below by the perimeter $P(E)$ of the boundary $\delta E$ of some region $E \subset (0,V_1)\times (0,\infty)$. The area of $E$, $A(E)$, satisfies $v= A(E)$ and $\delta E$ is a connected boundary curve along which $v_2$ is nonincreasing and $v_1$ is nondecreasing. The enclosed region $E$ is on the lower left of $\delta E$.  Hence

 $$P(E)\leq I_{S_{r_1}^1\times (S_{r_2}^1 \times \re^n)}(v)$$
 
where
 
\begin{equation}
\label{PE}
P(E) = \int_{\delta E} \sqrt{ h^2(v_2) dv_1^2 + 2^2 dv_2^2}.
\end{equation}
and the area of the region $E$ is given by
$$A(E) = \int \int_E dv_1 \, dv_2.$$

Since each term in the square root of eq. (\ref{PE})  is non-negative, we have 

\begin{equation}
\label{PE1}
P(E) \geq  2 \int_{\delta E}  dv_2
\end{equation}

and
\begin{equation}
\label{PE2}
P(E) \geq  \int_{\delta E}   h(v_2)   dv_1. 
\end{equation}

Now, using the hypothesis $0<\theta_m<\beta_n(r_2)<\theta_M$, we have from eq. (\ref{PE1})
  $$P(E) \geq  2 (\beta_n(r_2) -\theta_m), $$

and from eq. (\ref{PE2}),

  $$P(E) \geq \min_{v_2} \{h(v_2)\}  \int_{\delta E}      dv_1  \geq I_{S^1\times \re^n} (\theta_m)  \ \ V_1. $$

That is,  

$$P(E)\geq \max \{2 (\beta_n(r_2) -\theta_m), (V_1) \  I_{S_{r_2}^1\times \re^n} (\theta_m) \}= \max \{2 (\beta_n(r_2) - \theta_m), (V_1) \ I_{\re^{n+1}} (\theta_m) \}$$

\noindent where the last  equality follows from the fact that $ I_{S_{r_2}^1\times \re^n} (\theta_m)= I_{\re^{n+1}} (\theta_m)$
(since $\theta_m<\beta_n(r_2)$). Hence, for the isoperimetric region $\Omega$ we have,
$$ \max \{2 (\beta_n(r_2) -\theta_m), V_1 \ I_{  \re^{n+1}} (\theta_m) \}\leq P(E)\leq V^{n+1}(\partial \Omega).$$

By Remark \ref{solution}, there is a unique $\theta^*\in (0, \beta_n(r_2))$ such that satisfies
$2(\beta_n(r_2)- \theta^*)= V_1 I_{\re^{n+1}}(\theta^*),$ which is eq. (\ref{alg1}).
Note also that, as functions of $\theta_m$, $2 (\beta_n(r_2) -\theta_m)$ is decreasing while  $V_1 \ I_{  \re^{n+1}} (\theta_m)$ is increasing.  This yields 
 
 \begin{equation}
 \label{alga}
 V_1 \ I_{  \re^{n+1}} (\theta^*) \leq \max \{2 (\beta_n(r_2) -\theta_m), V_1 \ I_{  \re^{n+1}} (\theta_m) \} \leq P(E)\leq V^{n+1}(\partial \Omega),
 \end{equation}

\noindent regardless of the value of $\theta_m$. One may obtain a similar result, 
 \begin{equation}
\label{algb}
V^1(S^1_{r_2}) \ I_{  \re^{n+1}} (\sigma^*) \leq P(E)\leq V^{n+1}(\partial \Omega),
\end{equation}

\noindent being $\sigma^*\in (0, \beta_n(r_1))$, such that satisfies
eq. (\ref{alg2}), by following the same analysis for the product of $(S^1_{r_2},dt^2)$ with $(S^1_{r_1}\times \re^n,ds^2+g_E)$ and using the hypothesis $0<\sigma_m<\beta_n(r_1)<\sigma_M$. 

Since both eqs. (\ref{alga}) and  (\ref{algb}) occur, the conclusion of the Lemma follows.

\end{proof}

%%%%%%%%%%%%%%%%%%%%%%%%%%%%%%%%%%%%%%%%%%%%%%%%%%%%%%%%%%%%%%%%%%%%%%%%%%%%%%%%%%%%%%%%%%%%%%%%%%%%%%%%%%%%%%%%%%%%%%%%%%%%%%%%%%%%%%%%%%%%%%%%%%%%%%%%%%%%%%%%%%%%%%%%%%%%%%%%%%%%%%%%%%%%%%%%%%%%%%%%%%%%%%%%%%%%%%%%%%%%%%%
We now prove some lower bounds  for $V^{n+1}(\partial \Omega)$ for  the case where $0<\theta_m<\beta_n(r_2)<\theta_M$ and $0<\sigma_m<\beta_n(r_1)<\sigma_M$ occur.
\begin{Lemma}
	\label{4}
	
	Suppose  $0<\theta_m<\beta_n(r_2)<\theta_M$, and $0<\sigma_m<\beta_n(r_1)<\sigma_M$. Then
	
\begin{equation}
\label{7a}
I_{S^1_{r_2}\times \re^{n+1}}(v)- 2\beta_n(r_2) \leq I_{S^1_{r_1}\times S^1_{r_2}\times \re^n}(v)\end{equation}
	and
\begin{equation}
\label{7b}I_{S^1_{r_1}\times \re^{n+1}}(v)- 2\beta_n(r_1) \leq  I_{S^1_{r_1}\times S^1_{r_2}\times \re^n}(v)\end{equation}

\end{Lemma}

\begin{proof}
	We construct a new closed region $\Omega^* \subset [t_m, t_{m+2 \pi r_1}] \times S_{r_2}^1\times \re^{n} \subset \re \times S_{r_2}^1\times \re^{n}$, with $t_m$ such that $F_1(t_m)= \theta_m >0$, in the following way. For $t \in (0, 2 \pi r_1)$ let $\Omega^*_{t_m+t} = \Omega_{t}$. Also, let $\Omega^*_{t_m+2 \pi r_1} = \Omega_{t_m}$ and  $\Omega^*_{t_m} =\Omega_{t_m}$.
	
	Note that  $V^{n+2}( \Omega^*)=V^{n+2}( \Omega)$ and $V^{n+1}( \partial \Omega^*)=V^{n+1}(\partial \Omega)+ 2 V^{n+1}(\Omega_{t_m})=V^{n+1}(\partial \Omega)+ 2  \theta_m $.
		
	Let $v=V^{n+2}(\Omega)$. Since $\Omega^*$ is actually a closed set in $[t_m, t_{m+2 \pi r_1}] \times S_{r_2}^1\times \re^{n} \subset \re \times S_{r_2}^1\times \re^{n}$, it follows that 
	$$I_{S_{r_2}^1\times \re^{n+1}}(v)\leq V^{n+1}(\partial \Omega^*)=V^{n+1}(\partial \Omega)+ 2   \theta_m $$
	
	Since $   \theta_m < \beta_n(r_2)$, and $\Omega$ is isoperimetric, we have 
	
	$$I_{S^1_{r_2}\times \re^{n+1}}(v)- 2\beta_n(r_2)<I_{S^1_{r_2}\times \re^{n+1}}(v)- 2\theta_m \leq V^{n+1}( \partial \Omega)= I_{S^1_{r_1}\times S^1_{r_2}\times \re^n}(v).$$%\leq f_n(v).$$
	
%		It follows from eq. (\ref{Is1}) that 	$I_{S^1_{r_1}}\leq 	I_{S^1_{r_2}}$ and $ \beta_n(r_1)\leq \beta_n(r_2)$, since $r_1\leq r_2$. Hence,% from (\ref{r1}) and (\ref{r2}), 
	
%		$$I_{S^1_{r_1}\times \re^{n+1}}(v)- 2\beta_n(r_2) \leq  I_{S^1_{r_1}\times S^1_{r_2}\times \re^n}(v).$$

	Similarly,  using  $0<\sigma_m<\beta_n(r_1)<\sigma_M$, one may embed $\Omega$ in $S_{r_1}^1\times \re^{n+1}$ by replacing $\Omega_{s_m}$ with 2 copies of $\Omega_{s_m}$, where $F(s_m)=\sigma_m$. Following the same argument as before, one gets
	
		$$I_{S^1_{r_1}\times \re^{n+1}}(v)- 2\beta_n(r_1) \leq V^{n+1}(\partial \Omega)= I_{S^1_{r_1}\times S^1_{r_2}\times \re^n}(v).$$

\end{proof}

We now prove a straightforward lemma that will be useful.

\begin{Lemma}
	\label{algebraic}
	Let $a, b>0$, $n\in \mathbb{N}, n>1$. Consider  the function $\varphi(x)=\  x^{\frac{n}{n+1}} -a\  x^{\frac{n-1}{n}}$.  
	There is a unique $x_0>0$ such that $\varphi(x_0)=0$ and $\varphi(x)>0$ for $x>x_0$.
	\noindent	There is a unique  $x_1>0$ such that $\varphi(x_1)=b$ and $\varphi(x)>b$ for $x>x_1$. Moreover $x_0<x_1$.
\end{Lemma}

\begin{proof}
	For the first claim, we note that for $x>0$, $\varphi'(x)=0$ if and only if $$x= \left(\frac{(n-1)(n+1)}{n^2} \ a\right)^{n(n+1)}.$$
	Note that $\varphi(0)=0<b$, and  $\varphi(x)$ is decreasing for $x>0$ until $x_1= (\frac{(n-1)(n+1)}{n^2} \ a)^{n(n+1)}>0$ and increasing after that. This implies the first and second claims.
	
	For the third claim it suffices to remark that $\varphi(x)$ is still increasing after $x_0$ and that $\varphi(x_1)<0=\varphi(x_0)<b=\varphi(x_1)$; since $b>0$. It follows that $x_0<x_1$.
	
\end{proof}

\begin{Remark}
	\label{algebraic2}	
	
	Note that since $r_1$ and $r_2$ are fixed, equations (\ref{algc}) and (\ref{algd}) are algebraic equations on $v$ of the type $\    v^{\frac{n}{n+1}} -a\  v^{\frac{n-1}{n}}=b$, where $a, b, n>0$.
	By Lemma \ref{algebraic}  they have  unique solutions $v_i^{**}>0$; and $I_{S^1_{r_i}\times \re^{n+1}}(v)- f_n(v) > 2 \beta_n(r_i)$ for $v>v_i^{**}$, for $i=1,2$. Also,   $I_{S^1_{r_i}\times \re^{n+1}}(v)- f_n(v)=0$ has a unique solution $v_{0_i}>0$, and 	$I_{S^1_{r_i}\times \re^{n+1}}(v)> f_n(v)$ for $v>v_{0_i}$, for $i=1,2$.   Lemma \ref{algebraic} also implies $v_{0_i}<v_i^{**}$.
\end{Remark}

We now prove Theorem \ref{T2}.
%	Let $(T^2,h) = (S^1_{r_1} \times S^1_{r_1}, dt^2+ds^2)$, $r_1 \leq r_2$.

\begin{proof}

Let $\Omega \subset S^1_{r_1} \times S^1_{r_2} \times \re^n$ be an isoperimetric region. Consider the functions $F_1, F_2$ and the values $\theta_m, \theta_M, \sigma_m, \sigma_M$ as before. 

We begin with the case of big volumes. Let  $v^{**}= \max\{a_n, b_n\}$, where $a_n$ is such that 
\begin{equation}
\label{algc}
I_{S^1_{r_1}\times \re^{n+1}}(a_n)- f_n(a_n) = 2\beta_n(r_1), 
\end{equation}
and  $b_n$ such that
\begin{equation}
\label{algd}
I_{S^1_{r_2}\times \re^{n+1}}(b_n)- f_n(b_n) = 2\beta_n(r_2). 
\end{equation}

By Remark \ref{algebraic2}, for $v>v^{**}$, 
\begin{equation}
I_{S^1_{r_i}\times \re^{n+1}}(v)- f_n(v) > 2\beta_n(r_i),
\end{equation}
for $i=1,2.$ Hence, Lemma \ref{4} excludes the case $0 < \theta_m < \beta_n(r) < \theta_M$ and $0 < \sigma_m < \beta_n(r) < \sigma_M$. Remark \ref{algebraic2} also states that $v^{**} >v_0$, where $v_0=\max\{v_{0_1},v_{0_2}\}$ and $v_{0_i}$ is the unique $v$ such that $I_{S^1_{r_i}\times \re^{n+1}}(v) = f_n(v).$
This implies that for $v>v^{**}>v_0$ 
\begin{equation}
\label{excludes}
I_{S^1_{r_i}\times \re^{n+1}}(v) > f_n(v).
\end{equation}

Since $I_{S^1_{r_1} \times S^1_{r_2} \times \re^{n}}(v) \leq f_n(v)$ for all $v\geq0$, inequality (\ref{excludes})  excludes the following cases, if $V^{n+2}(\Omega)>v^{**}$:
\begin{enumerate}
	\item $\theta_m=0$ or $\sigma_m = 0$, by Lemma \ref{vol0}.
	\item $\theta_M < \beta_n(r_2)$ or $\theta_M < \beta_n (r_1)$, by Lemma \ref{vol1}.
\end{enumerate}

Thus, the only case left if $v=V^{n+2}(\Omega)>v^{**}$ is $\theta_m \geq \beta_n(r_2)$ or $\sigma_m \geq \beta_n(r_1)$, which implies $I_{S^1_{r_1} \times S^1_{r_2} \times \re^{n}}(v) = f_n(v)$
by Lemma \ref{vol2}.

We now treat the case of small volumes.

Let $v_s = \min \{V^{n+2} (S^1_{r_1} \times B^{n+1}_{\pi r_2}), V^{n+2} (B^{n+2}_{\pi r_1}) \}$.

Isoperimetric regions in $S_{r_1}^1\times \re^{n+1}$ are either regions of the type $B^{n+2}_{R}$ or $S_{r_1}^1\times B^{n+1}_R$, which are realizable in 
$S_{r_1}^1\times S_{r_2}^1 \times \re^{n+1}$ if $R < \pi r_1 \leq \pi r_2$.
This implies  
$I_{S_{r_1}^1\times S_{r_2}^1 \times \re^{n+1}} (v) \leq I_{S_{r_1}^1 \times \re^{n+1}}(v)$ for $v \leq v_s$.

Note that for $R < \pi r_2$, $S^1_{r_1} \times B^{n+1}_{R}$ is a closed region in $S^1_{r_1} \times S^1_{r_2} \times \re^n$. 

Hence for $v < v_s$ 
$$I_{S^1_{r_1} \times S^1_{r_2} \times \re^{n}}(v) \leq I_{S^1_{r_1} \times \re^{n+1}}(v).$$

Lemma \ref{vol0}  implies  that if $\theta_m=0$ or $\sigma_m=0$, then for $v<v_s$,
\begin{equation}
\label{lemmas}
I_{S^1_{r_1} \times \re^{n+1}}(v) \leq I_{S^1_{r_1} \times S^1_{r_2} \times \re^{n}}(v) \leq I_{S^1_{r_1} \times \re^{n+1}}(v).
\end{equation}

By Lemma \ref{vol1},  for $v<v_s$, these inequalities are also satisfied if $\theta_M \leq \beta_n(r_2)$ or $\sigma_M \leq \beta_n(r_1)$.

Note also that for $v < \min \{ v_s, v_{0_1}\}$, Lemma \ref{vol2} excludes the case $\theta_m \geq \beta_n(r_2)$ or $\sigma_m \leq \beta_n(r_1)$. Otherwise we would have 

$$I_{S^1_{r_1} \times \re^{n+1}}(v) > f_n(v) = I_{S^1_{r_1} \times S^1_{r_2} \times \re^{n}}(v) < I_{S^1_{r_1} \times \re^{n+1}}(v).$$

Finally, let $c_n$ be such that
$$ I_{S^1_{r_1} \times \re^{n}}(c_n)=K^*,$$
where $K^*$ is the constant defined in Lemma \ref{tm}.

Let $v_n^* = \min \{v_s, c_n, v_{0_1}\}$. Then, for $v<v_n^*$, 
$$I_{S^1_{r_1} \times S^1_{r_2} \times \re^{n}}(v) \leq I_{S^1_{r_1} \times \re^{n+1}}(v) < K^*.$$
 By Lemma \ref{tm}, this implies that  the case $0 \leq \theta_m \leq \beta_n(r_2) \leq \theta_M$ and $0 \leq \sigma_m \leq \beta_n(r_1) \leq \sigma_M$ is excluded for $v<v_n^*$.

We conclude that for $v<v_n^*$,
$$I_{S^1_{r_1} \times S^1_{r_2} \times \re^{n}}(v)= I_{S^1_{r_1} \times \re^{n+1}}(v).$$
\end{proof}

We now use these results to compute explicit lower bounds for the isoperimetric profile of a manifold of the type $(T^2\times \re^2, h_2+g_2)$.
\begin{figure}
	
	\includegraphics[scale=.3]{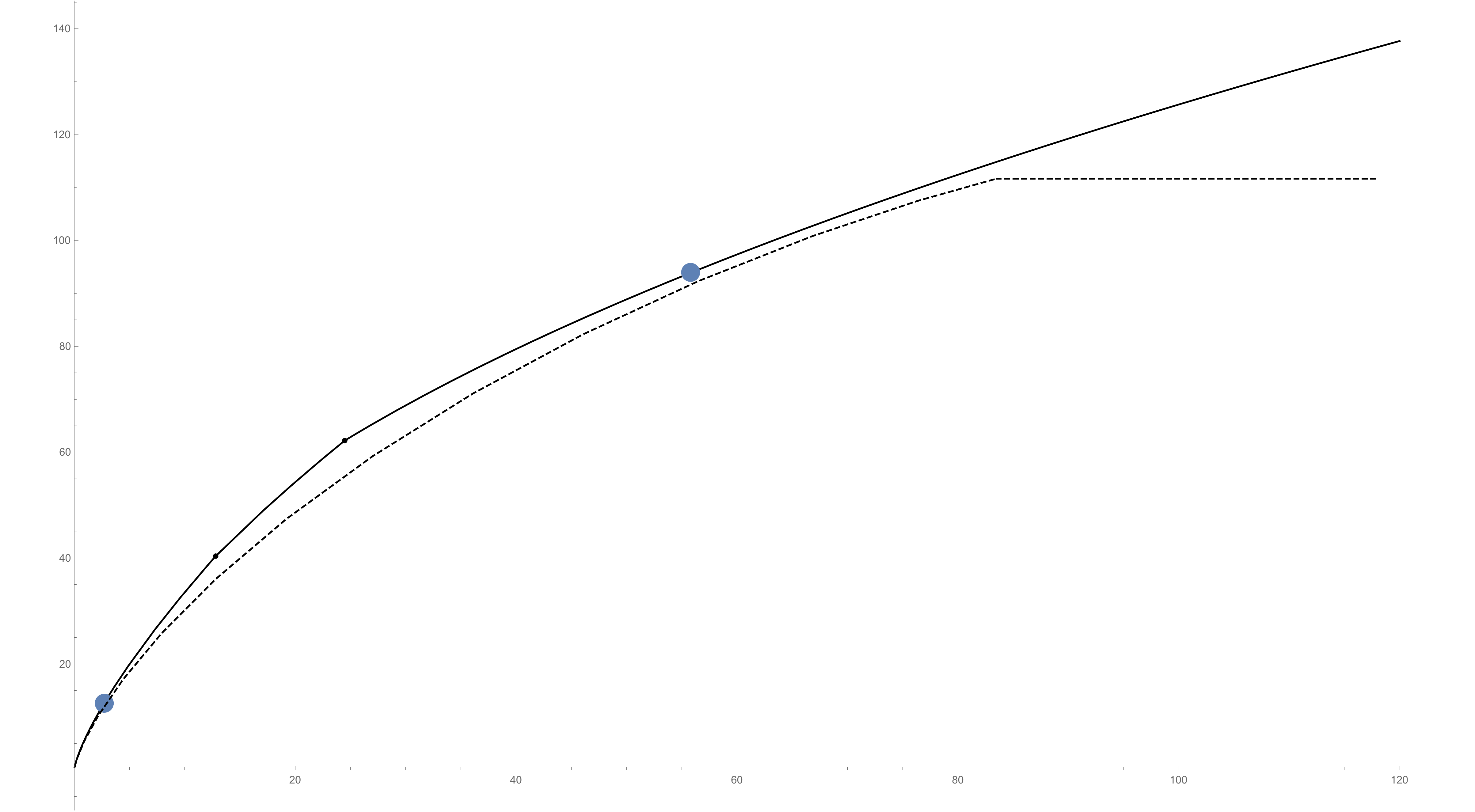}
	\caption{$I_{scp}(v)$ profile (solid line) is an upper bound for $I_{T^2\times \re^2}(v)$. The profile $I_{S_3^3\times \re}(v)$ (dashed line) is a lower bound for $I_{T^2\times \re^2}(v)$. Before $v_2^*\approx 2.7$ and after $v^{**}\approx 55.8$, the isoperimetric profile of $(T^2\times \re^2)$ is equal to the $I_{scp}(v)$ profile.}
	\label{isopa}
\end{figure}

\begin{Example}
	\label{example}
 Let $(\re^2,g_E)$ be the 2-dimensional Euclidean space and $(T^2,h)$ be the 2-Torus with a flat metric, isometric to $\re / \Gamma$ with $\Gamma$ the lattice generated by $\{\{2{\sqrt\pi},0\},\{0,2{\sqrt \pi}\}\}$.	Using Theorem \ref{T2} and its proof, one may make explicit estimates of  the isoperimetric profile of $(T^2\times \re^2, h+g_2)$.
\end{Example}

We are in the case $r_1=r_2=\frac{1}{\sqrt\pi}$, $n=2$ of Theorem \ref{T2}.   By solving eqs. (\ref{algc}) and  (\ref{algd})  we get $v_2^{**}\approx  55.84$. We also compute $v_2^*\approx 2.70$, using equations (\ref{K}),  (\ref{alg1}) and (\ref{alg2}). The $I_{scp}(v)$ profile, given by $I_{scp}(v)=\min\{I_{S_{r_1}^1\times \re^{3}}(v), f_2(v)\}$, is an upper bound for $I_{T^2 \times \re^{n}}(v)$, moreover, if  $v\leq v_2^*$ or   $v\geq v_2^{**}$, then  $I_{T^2 \times \re^{2}}(v)=I_{scp}(v)$. The solid line graphic of figure \ref{isopa} is the graphic of  $I_{scp}(v)$.  In this case $f_2(v)=4 \pi   \sqrt{v}$.

One may compute  lower bounds for the volumes between $v_2^*$ and $v_2^{**}$. First, since the  Ricci curvature of  $(T^2\times \re^2,h+g_E)$ is non-negative, it follows from a result by V. Bayle (\cite{Bayle}, p. 52) that the isoperimetric profile is concave.  This implies that a line joining the points $(v^{*},  I_{scp}(v^{*}))$ and  $(v^{**},  I_{scp}(v^{**}))$ is also a lower bound for $I_{T^2\times \re^2}(v)$.

A better lower bound for $I_{T^2\times \re^2}(v)$ may be computed in the following way. Since the isoperimetric profiles of $(S^2,g_2)$ and $(T^2,h)$ are known explicitly, it is straightforward to check $I_{S^2} \leq I_{T^2}$. Here, $g_2$ is the round metric with radius $r=1$. Since $S^2$ is a model metric, it follows from the Ros symmetrization Theorem \cite{Ros}, that $I_{S^2\times \re^2}\leq I_{T^2\times \re^2}$.  On the other hand, it was proved in  section 2.1 of \cite{Petean2} that  $I_{S_3^3\times \re}\leq I_{S^2\times \re^2}$, where $(S_3^3,g_3)$ is the 3-sphere with the round metric and radius $r=3$.  It follows that 
$I_{S^3_3\times \re} \leq I_{T^2\times \re^2}$. The isoperimetric profile of $I_{S^3_3\times \re}$ was computed in \cite{Pedrosa} and its graphic corresponds to the the dashed graphic of figure \ref{isopa}. Moreover, using that the isoperimetric profile of  $(T^2\times \re^2,h+g_E)$ is concave, it follows that any line joining the point $(v^{**}, I_{scp}(v^{**}))$ and the graphic of  $I_{S_3^3\times \re}(v)$ is also a lower bound for $I_{T^2\times \re^2}(v)$. Similarly, any line joining the point $(v^{*}, I_{scp}(v^{*}))$ and the graphic of  $I_{S_3^3\times \re}(v)$ is also a lower bound for $I_{T^2\times \re^2}(v)$.

  These lower bounds gives us a fair idea of the shape of $I_{T^2\times \re^2}$  in the interval $(v_2^*,v_2^{**})$ and are illustrated in figure \ref{isopprofT2R2}.% Before $v_2^*$, $I_{T^2\times \re^2}(v)= I_{S^1\times \re^{3}}(v)=2( 2^{3/4} )\sqrt{\pi} \ v^{3/4}$ and after $v_2^{**}$, $I_{T^2\times \re^2}(v)=f_2(v)=4 \pi   \sqrt{v}$. 
%%%%%%%%%%%%%%%%%%%%%%%%%%%%%%%%%%%%%%%%%%%%%%%%%%%%%%%%%%%%%%%%%%%%%%%%%%%%%%%%%%%%%%%%%%%%%%%%%%%%%%%%%%%%%%%%%%%%%%%%%%%%%%%%%%%%%%%%%%%%%%%%%%%%%%%%%%%%%%%%%%%%%%%%%%%%%%%%%%%%%%%%%%%%%%%%%%%%%%%%%%%%%%%%%%%%%%%%%%%%%%%%%%%%%%%%%%%%%%%%%%%%%%%%%%%%%%%%%%%%%%%%%%%%%%%%%%%%%%%%%%%%%%%%%%%%%%%%%%%%%%%%%%%%%%%%%%%%%%%%%%%%%%%%%%%%%%%%%%%%%%%%%%%%%%%%%%%%%%%%%%%%%%%%%%%%%%%%%%%%%%%%%%%%%%%%%%%%%%%%%%%%%%%%%%%%%%%%%%%%%%%%%%%%%%%%%%%%%%%%%%%%%%%%%%%%%%%%%%%%%%%%%%%%%%%%%%%%%%%%%%%%%%%%%%%%%%%%%%%%%%%%%%%%%%%%%%%%%%%%%%%%%%%%%%%%%%%%%%%%%%%%%%%%%%%%%%%%%%%%%%%%%%%%%%%%%%%%%%%%%%%%%%%%%%%%%%%%%%%%%%%%%%%%%%%%%%%%%%%%%%%%%%%%%%%%%%%%%%%%%%%%%%%%%%%%%%%%%%%%%%%%%%%%%%%%%%%%%%%%%%%%%%%%%%%%%%%%%%%%%%%%%%%%%%%%%%%%%%%%%%%%%%%%%%%%%%%%%%%%%%%%%%%%%%%%

%%%%%%%%%%%%%%%%%%%%%%%%%%%%%%%%%%%%%%%%%%%%%%%%%%%%%%%%%%%%%%%%%%%%%%%%%%%%%%%%%%%%%%%%%%%%%%%%%%%%%%%%%%%%%%%%%%%%%%%%%%%%%%%%%%%%%%%%%%%%%%%%%%%%%%%%%%%%%%%%%%%%%%%%%%%%%%%%%%%%%%%%%%%%
%%%%%%%%%%%%%%%%%%%%%%%%%%%%%%%%%%%%%%%%%%%%%%%%%%%%%%%%%%%%%%%%%%%%%%%%%%%%%%%%%%%%%%%%%%%%%%%%%%%%%%%%%%%%%%%%%%%%%%%%%%%%%%%%%%%%%%%%%%%
%%%%%%%%%%%%%%%%%%%%%%%%%%%%%%%%%%%%%%%%%%%%%%%%%%%%%%%%%%%%%%%%%%%%%%%%%%%%%%%%%%%%%%%%%%%%%%%%%%%%%%%%%%%%%%%%%%%%%%%%%%%%%%%%%%%%%%%%%%%
\section{The isoperimetric profile of $T^k\times \re^n$}

One may  follow the arguments of the last section in order to understand the isoperimetric profile of $T^k\times \re^n$ for small  and big volumes. In this section we present the proof of Theorem \ref{T3}, that is, the case  $k=3$. The more general case, $2\leq k\leq5$, $2\leq n\leq 7-k$, is similar.%Since for the arguments we also need an understanding of $S^1\times \re^n$, the results for the isoperimetric profile of $T^3\times \re^n$ are be valid for $2\leq n\leq 6$.

Let $(T^3,h_3)= (S_{r_1}^1\times S_{r_2}^1 \times S_{r_3}^1, ds_1^2+ds_2^2+ds_3^3)$.
Let $\Omega$ be an isoperimetric region in $(T^3\times \re^n,h_3+g_E)$. We parameterize $S_{r_3}^1$, by $[0,2 \pi r_3)$ and consider  slices $\Omega_{t}$, $t\in [0,2\pi r_3)$:

$$\Omega_{t}=\Omega \cap(  S^1_{r_1}\times S^1_{r_2}\times \{t\} \times  \re^n).$$

\noindent Then for each slice $\Omega_{t}$,   we may compute its $(n+2)$-volume and define a function $F:[0,2 \pi \ r_3]\rightarrow \re$, by $F(t)=V^{n+2}(\Omega_{t})$ and $F(2 \pi r_3)=V^{n+2}(\Omega_0)$.

Note that  $F$ is continuous.
Let $\eta_{m}$ and $\eta_{M}$ denote the minimum and maximum values of $F$, respectively.

\begin{Lemma}
	\label{kvol0}
	If  $\eta_{m} =0$, then $ I_{ S^1_{r_1}\times S^1_{r_2}\times \re^{n+1}}(V^{n+3}(\Omega)) \leq V^{n+2}(\partial \Omega)$.
\end{Lemma}
\begin{proof}
	Suppose	$\eta_{m}=0$.  Let $t_0 \in [0, 2 \pi r_3]$ be such that $F(t_0)=\eta_{m}=0$.
	%Then $\Omega_{t_0}$ has  zero $(n+1)$-volume in $\{t_0\}\times S_{r_2}^1\times \re^n$. 
	We construct a new closed region $\Omega^* \subset  S_{r_1}^1 \times S_{r_2}^1\times [t_0,t_0+2 \pi r_3]\times\re^{n}$. We denote $\Omega^* \cap(  S^1_{r_1}\times S^1_{r_2}\times \{t\} \times  \re^n)$ by $\Omega^*_{t}$.
	For $t\in [0, 2\pi r_3)$, let $\Omega^*_{t_0+t}=\Omega_t$. Also, let $\Omega^*_{t_0+2\pi r_3}=\Omega_{t_0}$.  $\Omega^*$ is a closed region by continuity of $F$. Also, 	since  $V^{n+2}(\Omega_{t_0})=0$ we have  $V^{n+2}(\partial \Omega)=V^{n+2}(\partial\Omega^*)$ and $V^{n+2}(\Omega)=V^{n+2}(\Omega^*)$.  Hence 
		\begin{equation}
	\label{firstb}
	I_{ S^1_{r_1}\times S^1_{r_2}\times \re^{n+1}}(V^{n+3}(\Omega))=I_{ S^1_{r_1}\times S^1_{r_2}\times \re^{n+1}}(V^{n+3}(\Omega^*))\leq V^{n+2}(\partial \Omega^*) = V^{n+2}(\partial\Omega).
	\end{equation}
%	Finally, since $r_1\leq r_i$, 
\end{proof}

Let $w_n^*=\min\{v_n^*, \beta_{n+1}(r_1)\}$, where $v_n^*$ is as in the hypothesis of Theorem \ref{T3} and $\beta_{n+1}(r_1)$ as in eq. (\ref{Is1}). Hence, for $v<w^*$ we have 
\begin{equation}
\label{w}
I_{S^1_{r_1}\times S^1_{r_2}\times \re^n}(v)=I_{S^1_{r_1} \times \re^{n+1}}(v)=I_{\re^{n+2}}(v)
\end{equation}
\begin{Lemma}
	\label{kvol1}
	If $\eta_M \leq w^*_n$, then $ I_{S^1_{r_3}\times \re^{n+2}}(V^{n+3}(\Omega)) \leq V^{n+2}(\partial \Omega)$.
\end{Lemma}
\begin{proof}
	We symmetrize $\Omega$ by constructing a  region $\Omega^*$ as in the proof of Lemma \ref{vol1}.
	We denote $\Omega^* \cap(  S^1_{r_1}\times S^1_{r_2}\times \{t\} \times  \re^n)$ by $\Omega^*_{t}$.
	Let  $\Omega^*_t=\{t\}\times B_{R(t)}^{n+2}$, where $R(t)>0$ is such that  $V^{n+2}(\{t\}\times B_{R(t)}^{n+2})=V^{n+2}(\Omega_t)$.
	
	Note that 
	$$V^{n+2}(\Omega^*_{t})= V^{n+2}(\Omega_t)$$
	
	\noindent and since $\eta_M<w_n^*$, each region $\{t\}\times B_{R(t)}^{n+1}$ is isoperimetric in 
	$\{t\}\times S_{r_1}^1\times S_{r_2}^1\times \re^n$, that is $V^{n+1}(\partial \Omega_t^*)\leq V^{n+1}(\partial \Omega_t)$.
	
	Arguing as in the proof of the Ros product Theorem (\cite{Ros}), we get
	
	$V^{n+3}(\Omega^*)= V^{n+3}(\Omega)$ and $V^{n+2}(\partial \Omega^*)\leq V^{n+2}(\partial \Omega).$
	
	This implies
	$$I_{S^1_{r_3}\times \re^{n+2}}(V^{n+3}(\Omega))=I_{S^1_{r_3}\times \re^{n+2}}(V^{n+3}(\Omega^*))\leq  V^{n+2}(\partial \Omega^*) \leq V^{n+2} (\partial \Omega).$$

\end{proof}

Let $g_n$ be the function given by $g_n (v) = V^{n+2}(\partial(S_{r_1}^1\times S_{r_2}^1\times S_{r_3}^1\times B^n_{R}))$, where $R$ is such that $V^{n+3}(S_{r_1}^1\times S_{r_2}^1\times S_{r_3}^1\times B^n_{R}))=v$.

\begin{Lemma}
	\label{kvol2}
	Suppose $\eta_m \geq v_n^{**}$. 
	If $V^{n+3}(\Omega)<\beta_n(r_3) V^1(S_{r_1}^1)V^1(S_{r_2}^1)$, then
	
	$$ I_{S_{r_1}^1\times S_{r_2}^1\times \re^{n+1}}(V^{n+3}(\Omega)) \leq V^{n+2}( \partial \Omega).$$ 
	On the other hand, if $V^{n+3}(\Omega)\geq \beta_n(r_3) V^1(S_{r_1}^1)V^1(S_{r_2}^1)$, then %$S_{r_1}^1\times S_{r_2}^1\times S_{r_3}^1\times B^n_{R}$, where $R$ is such that $V^{n+3}(S_{r_1}^1)\times S_{r_2}^1\times S_{r_3}^1\times B^n_{R})=V^{n+3}(\Omega)$, is an isoperimetric region;
	%and hence
		\begin{equation}
	\label{kf24}
	 V^{n+2}( \partial \Omega)=g_n(V^{n+3}(\Omega)) .
	 	\end{equation}
\end{Lemma}
\begin{proof}

	We construct a new region $\Omega^*$. We denote $\Omega^* \cap(  S^1_{r_1}\times S^1_{r_2}\times \{t\} \times  \re^n)$ by $\Omega^*_{t}$. Let 
	$\Omega_t^*=\{t\}\times S^1_{r_1}\times S^1_{r_2}\times B^n_{R(t)}$, with $R(t)$ such that $V^{n+2}(S^1_{r_1}\times S^1_{r_2}\times B^n_{R(t)})=V^{n+2}(\Omega_t)$.
	Since 
	$\eta_m > v^{**}_n$, regions of the type $S^1_{r_1}\times S^1_{r_2}\times B^n_{R(t)}$ are isoperimetric in $S^1_{r_1}\times S^1_{r_2}\times \re^n$. 
	Arguing as in the proof of the Ros product Theorem (\cite{Ros}), we get
    $V^{n+3}(\Omega^*)= V^{n+3}(\Omega)$ and $V^{n+2}(\partial \Omega^*)\leq V^{n+2}(\partial \Omega).$

	We now symmetrize $\Omega^*$. 
	Let  $(\Omega^*)^p = \Omega^* \cap (\{p\} \times S^1_{r_3} \times \re^n)$ where $p \in S^1_{r_1}\times S^1_{r_2}$. 
	
	Let $p, q \in S^1_{r_1}\times S^1_{r_2}$. 
	Note that $(\Omega^*)^p = \bigcup_{t \in S^1_{r_3}} \left ( \Omega_t^* \cap ( \{ p \} \times \{ t \} \times B^n_{R(t) })\right )$ and $(\Omega^*)^q = \bigcup_{t \in S^1_{r_3}} \left ( \Omega_t^* \cap ( \{ q \} \times \{ t \} \times B^n_{R(t)} )\right )$. Since $R(t)$ is the same on both slices we get,
	\begin{equation}
	\label{equalityvol1}
	V^{n+1}((\Omega^*)^p) = V^{n+1}((\Omega^*)^q).
	\end{equation}
	Since $p,q$ where arbitrary, this implies 
	\begin{equation}
	\label{equalityvol2}
	V^{n+3}(\Omega) = V^{n+1}((\Omega_p) V^1 (S_{r_1}) V^1 (S_{r_2}).
	\end{equation}

	If $V^{n+3}(\Omega)<\beta_n(r_3) V^1(S_{r_1}^1)V^1(S_{r_2}^1)$,
	then by eq. (\ref{equalityvol2}), $V^{n+1}((\Omega^*)^p) < \beta_n(r_3)$
	and hence balls $B^{n+1}_{R}$ are isoperimetric regions in $S^1_{r_3} \times \re^n$.
	
	We construct a new region $\Omega^{**}$ such that
	$(\Omega^{**})^p = \{ p \} \times B^{n+1}_{R}$, $p \in S_{r_1}^1 \times S_{r_2}^1$ with $R$ such that $V^{n+1} (B^{n+1}_{R}) = V^{n+1} ((\Omega^*)^p)$ (note that $R$ is independent of $p$, by eq. (\ref{equalityvol1})).
	
	Arguing as in the Ros Product Theorem, we then have $V^{n+3}(\Omega^{**}) = V^{n+3}(\Omega^{*})= V^{n+3}(\Omega)$ and $V^{n+2}(\partial \Omega^{**}) \leq V^{n+2}(\partial \Omega^{*}) \leq V^{n+2}(\partial \Omega)$. This implies the first part of the Lemma: 
	$$ I_{S_{r_1}^1\times S_{r_2}^1\times \re^{n+1}}(V^{n+3}(\Omega)) \leq V^{n+2}( \partial \Omega).$$ 
		
	On the other hand, if $V^{n+3}(\Omega)\geq\beta_n(r_3) V^1(S_{r_1}^1)V^1(S_{r_1}^1)$
	then by eq. (\ref{equalityvol2}), $V^{n+1}((\Omega^*)^p) \geq \beta_n(r_3)$
	and hence $S^1_{r_3}\times B^{n}_{R}$ are isoperimetric regions in $S^1_{r_3} \times \re^n$.
	
	We then construct a new region $\Omega^{**}$ such that
	$(\Omega^{**})^p = \{ p \} \times S^1_{r_3}\times B^{n}_{R}$, $p \in S_{r_1}^1 \times S_{r_2}^1$, with $R$ such that $V^{n+1} ( S^1_{r_3}\times B^{n}_{R}) = V^{n+1} ((\Omega^*)^p)$ ($R$ is independent of $p$, by eq. (\ref{equalityvol1})).
	
	Arguing as in the Ros Product Theorem, we get $V^{n+3}(\Omega^{**}) = V^{n+3}(\Omega^{*})= V^{n+3}(\Omega)$ and $V^{n+2}(\partial \Omega^{**}) \leq V^{n+2}(\partial \Omega^{*}) \leq V^{n+2}(\partial \Omega)$. This implies 
	$ g_n (V^{n+3}( \Omega))\leq V^{n+2}( \partial \Omega).$
	
	Being $\Omega$ isoperimetric, we conclude that $g_n (V^{n+3}( \Omega))= V^{n+2}( \partial \Omega).$

	%and the conclu Lemma follows.
	
\end{proof}

%We are now ready to prove the first part of Theorem \ref{T2}.
We now prove that the case $\eta_M>w_n^*$ cannot occur for small areas of $\Omega$.

\begin{Lemma}
	\label{tm4} Suppose that $\eta_M>w_n^*$.
	
	Then there is some $C^*>0$ such that  %for any isoperimetric region $\Omega$,
	$$V^{n+2}(\partial \Omega) > C^*.$$
	\noindent $C^*$  is independent of $\Omega$. In fact, it depends only on $r_1,r_2,r_3,n$ and is given by $$C^* = 2(w_n^* - \eta^*)$$
	where $\eta^* > 0$ satisfies
	\begin{equation}
	\label{eta1}
	\frac{1}{2} V(S^1_{r_3})I_{\re^{n+2}}(\eta^*)+ \eta^* = w_n^*
	\end{equation}
	
\end{Lemma}

\begin{proof}

	%Hence, for the proof, we will only check the case  that $0<t_m<\beta_n(r_2)<t_M$ and $0<s_m<\beta_n(r_1)<s_M$ occur.
	
	The proof is similar to that of lemma \ref{tm}.
\end{proof}
\begin{Remark}
	By remark \ref{solution}, a solution to equation (\ref{eta1}) exists and is  unique. %Let $\varphi(x)=a \  x^{\frac{n+1}{n+2}}+x$. Since $a,b>0$, we have $\varphi(0)<b$ and $\varphi(b)>b$; by continuity of  $\varphi(x)$, a solution $\varphi(x)=b$ exists in $(0,b)$. Uniqueness follows from the fact that $\varphi(x)$ is  increasing.
\end{Remark}

%%%%%%%%%%%%%%%%%%%%%%%%%%%%%%%%%%%%%%%%%%%%%%%%%%%%%%%%%%%%%%%%%%%%%%%%%%%%%%%%%%%%%%%%%%%%%%%%%%%%%%%%%%%%%%%%%%%%%%%%%%%%%%%%%%%%%%%%%%%%%%%%%%%%%%%%%%%%%%%%%%%%%%%%%%%%%%%%%%%%%%%%%%%%%%%%%%%%%%%%%%%%%%%%%%%%%%%%%%%%%%%
We now prove a lower bound  for the area of the region $\Omega$, $V^{n+2}(\partial \Omega)$, for  the case  $\eta_m< v_n^{**}$.
\begin{Lemma}
	\label{k4}
	
	Suppose  $0<\eta_m< v_n^{**}$. Then
	
	\begin{equation}
	\label{k7a4}
	I_{S^1_{r_1}\times S^1_{r_2} \times \re^{n+1}}(V^{n+3}( \Omega))- 2 v_n^{**} < V^{n+2}(\partial \Omega)
	\end{equation}
\end{Lemma}

\begin{proof}
	We construct a new closed region $\Omega^* \subset S_{r_1}^1\times S_{r_2}^1\times  [t_m, t_{m+2 \pi r_3}] \times \re^{n}\subset S_{r_1}^1\times S_{r_2}^1\times \re \times \re^{n}$, with $t_m$ such that $F(t_m)= \eta_m >0$, in the following way. We denote $\Omega^* \cap(  S^1_{r_1}\times S^1_{r_2}\times \{t\} \times  \re^n)$ by $\Omega^*_{t}$.  For $t \in (0, 2 \pi r_3)$ let $\Omega^*_{t_m+t} = \Omega_{t}$. Also, let $\Omega^*_{t_m+2 \pi r_3} = \Omega_{t_m}$ and  $\Omega^*_{t_m} =\Omega_{t_m}$.
	
	Note that  $V^{n+3}( \Omega^*)=V^{n+3}( \Omega)$ and $V^{n+2}( \partial \Omega^*)=V^{n+2}(\partial \Omega)+ 2 V^{n+2}(\Omega_{t_m})=V^{n+2}(\partial \Omega)+ 2  \eta_m $.
	 Since $\Omega^*$ is actually a closed set in $S_{r_1}^1\times S_{r_2}^1\times  \re \times \re^{n}$, it follows that 
	$$I_{S_{r_1}^1\times S_{r_2}^1\times \re^{n+1}}(V^{n+3}(\Omega))\leq V^{n+2}(\partial \Omega^*)=V^{n+2}(\partial \Omega)+ 2   \eta_m $$
	
	Since $   \eta_m <  v_n^{**}$, we have eq. (\ref{k7a4}).
	
	%$$I_{S^1_{r_1}\times S^1_{r_2}\times \re^{n+1}}(v)- 2 v_n^{**} \leq V^{n+1}( \Omega)$$
\end{proof}

We now prove Theorem \ref{T3}.
%	Let $(T^2,h) = (S^1_{r_1} \times S^1_{r_1}, dt^2+ds^2)$, $r_1 \leq r_2$.

\begin{proof}
	
	Let $\Omega \subset S^1_{r_1} \times S^1_{r_2}\times S^1_{r_3}  \times \re^n$ be an isoperimetric region. Consider the functions $F$, $g_n$ and the values $\eta_m, \eta_M$ as before. 
	
	We begin with the case of big volumes. By Lemma \ref{algebraic} and Remark \ref{algebraic2}, there exists $u_n^{**}$ such that for $v> u_n^{**}$ 
	$$I_{S^1_{r_1} \times S^1_{r_2} \times \re^{n+1}}(v)- 2 v_n^{**} > g_n(v).$$
	
	Being $\Omega$ isoperimetric, we have $V^{n+1}(\partial \Omega) \leq g_n(v)$.
	Hence, if $V^{n+3}(\Omega) > u_n^{**}$,
	\begin{equation}
	\label{eq101}
	I_{S^1_{r_1} \times S^1_{r_2} \times \re^{n+1}}(V^{n+3}(\Omega))- 2 v_n^{**} \geq V^{n+2}(\partial \Omega).
	\end{equation}
	
	Hence, eq. (\ref{eq101}) and Lemma \ref{k4} prevents $\eta_m < v_n^{**}$ from happening if $V^{n+3}(\Omega) > u_n^{**}$. This implies that for $V^{n+3}(\Omega) > u_n^{**}$, $\eta_m \geq v_n^{**}$ and by Lemma \ref{kvol2}, that $V^{n+2}(\partial \Omega)=g_n(V^{n+3}(\Omega)).$
	
	%%%%%%%%%%%%%%%%%%%%%%%%%%%%%%%%%%%%%%%%%%%%%%%%%%%%%%%%
	We now treat the case of small volumes.
	
	Suppose $V^{n+2}(\partial \Omega)<C^*$. 
	By Lemma \ref{tm4}, $\eta_{m}<w_n^*$. 	This implies, by Lemma \ref{kvol0} and \ref{kvol1}, that 
	$$\min \{ I_{S^1_{r_1}\times S^1_{r_2}\times \re^{n+1}(v)}, I_{S^1_{r_3} \times \re^{n+2}}(v) \} \leq I_{S^1_{r_1}\times S^1_{r_2}\times S^1_{r_3} \times \re^{n}}(v).$$
	On the other hand, $I_{S^1_{r_1} \times \re^{n+2}}(v) \leq I_{S^1_{r_3}\times \re^{n+2}}(v)$ since $r_1 \leq r_3$ and $I_{S^1_{r_1} \times \re^{n+2}}(v) \leq I_{S^1_{r_1}\times S^1_{r_2} \times \re^{n+1}}(v)$, if $v \leq v_{n+1}^*$.

	Hence, if $V^{n+3}(\Omega)\leq v_{n+1}^*$ and $V^{n+2}(\partial \Omega)<C^*$, we have
	$$I_{S^1_{r_1} \times \re^{n+2}}(V^{n+3}(\Omega)) \leq I_{S^1_{r_1}\times S^1_{r_2}\times S^1_{r_3} \times \re^{n}}(V^{n+3}(\Omega)).$$
	
	Note also that if $v<V^1(S^1_{r_1})V^n(B_{r_2})$, then isoperimetric regions in $S^1_{r_1}\times \re^{n+2}$ are realizable in $S^1_{r_1}\times S^1_{r_2}\times S^1_{r_3} \times \re^{n}$. Hence, in this case,

$$ 	I_{S^1_{r_1} \times \re^{n+2}}(V^{n+3}(\Omega)) \geq I_{S^1_{r_1}\times S^1_{r_2}\times S^1_{r_3} \times \re^{n}}(V^{n+3}(\Omega)).$$
	
	Let $u_0>0$ be such that $I_{S^1_{r_1}\times \re^{n+2}}(u_0)=C^*$ and $u_n^*=\min\{u_0,v_{n+1}^*, V^1(S^1_{r_1})\  V^{n+2}(B^{n+2}_{r_2}))\}$
		
		We conclude that if $V^{n+3}(\Omega)\leq u^*_n$,
		$$ 	I_{S^1_{r_1} \times \re^{n+2}}(V^{n+3}(\Omega)) =I_{S^1_{r_1}\times S^1_{r_2}\times S^1_{r_3} \times \re^{n}}(V^{n+3}(\Omega)).$$
		
\end{proof}

\begin{comment}

\end{comment}

\end{document}